\documentclass[11pt,a4paper,reqno]{amsart}
\usepackage{txfonts}
\usepackage{mathrsfs}
\usepackage{bbm}
\usepackage[mathscr]{euscript}
\usepackage{amsmath,amsthm,amsfonts,amssymb}
\usepackage{latexsym}
\usepackage{enumerate}
\usepackage{colortbl}
\def\bc{\begin{center}}
\def\ec{\end{center}}
\def\be{\begin{equation}}
\def\ee{\end{equation}}
\def\F{\mathbb F}
\def\N{\mathbb N}

\def\S{\mathfrak S}
\def\R{\mathbb R}

\def\G{\mathbb{G}}
\usepackage{colortbl}
\def\P{\textsf{P}}
\newtheorem{theoremalph}{Theorem}

\newtheorem{lem}{Lemma}[section]

\newtheorem{dfn}[lem]{Definition}
\newtheorem{pro}[lem]{Proposition}
\newtheorem{thm}[lem]{Theorem}

\newtheorem{cor}[lem]{Corollary}
\theoremstyle{remark}
\newtheorem{rem}{Remark}
\numberwithin{equation}{section}
\newcommand\zu{[0,1]}
\newcommand\ep{\varepsilon}

\begin{document}
\title[Quantitative recurrence   in conformal IFS]
{{Quantitative recurrence properties in conformal iterated function systems}}

\author[S. Seuret and B. W. Wang]{S. Seuret$^\ast$ and B. W. Wang}

\address {St\'ephane SEURET, Universit\'e Paris-Est, LAMA (UMR 8050), UPEMLV, UPEC, CNRS, F-94010, Cr\'eteil, France}
\email{seuret@u-pec.fr}
\address{BaoWei Wang, School of Mathematics and Statistics, Huazhong University of Science and Technology, 430074 Wuhan, P. R. China.}
\email{bwei\_wang@hust.edu.cn}
\keywords {Quantitative recurrence property, Conformal iterated function systems, Hausdorff dimension.}
 \thanks{$^\ast$ Research partially supported by the ANR project MUTADIS, ANR-11-JS01-0009.  }

\subjclass[2010]  {37F35, 37D35, 28A80.}

\begin{abstract}
Let   $\Lambda$ be  a countable index set and  $\S=\{\phi_i: i\in \Lambda\}$ be a conformal iterated function system  on  $[0,1]^d$ satisfying the open set condition.   Denote by $J$ the attractor of $\S$. With each sequence $(w_1,w_2,...)\in \Lambda^{\N}$ is associated a unique point $x\in \zu^d$. Let $J^\ast$ denote the set of points of $J$ with unique coding, and define the mapping $T:J^\ast \to J^\ast$ by $Tx= T (w_1,w_2, w_3...) = (w_2,w_3,...)$.
In this paper, we consider the quantitative recurrence properties related to the dynamical system $(J^\ast, T)$. More precisely, let $f:[0,1]^d\to \mathbb{R}^+$ be a positive function  and   $$
R(f):=\{x\in J^\ast: |T^nx-x|<e^{-S_n f(x)}, \ {\text{for infinitely many}}\ n\in \N\},
$$ where $S_n f(x) $ is the $n$th Birkhoff sum associated with the potential $f$.  In other words, $R(f)$   contains the points $x$ whose orbits return close to $x$ infinitely often, with a rate  varying along   time. Under some conditions, we  prove that  the Hausdorff dimension of  $R(f)$  is given by $\inf\{s\ge 0:\P(T, -s(f+\log |T'|))\le 0\}$, where $\P$ is the pressure function and $T'$ is the derivative of $T$. We present some applications of the main theorem to Diophantine approximations.
\end{abstract}

 \maketitle

\section{Introduction}
Diophantine analysis in a dynamical system yields an important way to understand the asymptotic behavior  of the orbits of the system.
The pioneer work of Poincar\'{e} states that in a measure theoretic dynamical system, 
almost all orbits will return to the initial point infinitely many times. In a metric space $(X,d)$ endowed with a transformation  $T:X\to X$ and a $T$-invariant Borel probability measure $\mu$, for $\mu$-almost all $x\in X$, one has
 \begin{equation}
 \label{eq1}
\liminf_{n\to \infty}d(T^nx, x)=0.
\end{equation}
It should be emphasized that Poincar\'{e}'s recurrence theorem is only qualitative in nature; it does not address the problem of the rate of convergence in \eqref{eq1}. This leads to the study on the so-called quantitative properties of Poincar\'{e}'s recurrence theorem \cite{Bo} or a type of shrinking target problems \cite{HV95}.

One can distinguish three ``shrinking target" problems: \begin{itemize}
\item {\bf Shrinking target problems with given targets}: let $\{z_n\}_{n\ge 1}$ be  a sequence of elements in $X$ and $\psi:\N\times X \to \R^+$. One  is interested in the points whose orbits are well approximated by the   sequence $\{z_n\}$ with   speed $\psi$, i.e. the set $$
S(T, \psi):=\Big\{x\in X: |T^nx-z_n|<\psi(n,x), \ {\text{i.o.}}\ n\in \N\Big\}.
$$
This can be interpreted as a dynamical version of the classic Diophantine approximation \cite{Sp}. References on this subject include Chernov \& Kleinbock \cite{ChK}, Maucourant \cite{Ma}, Galatolo \& Kim \cite{GK}, Tseng \cite{Tseng} and Fern\'{a}ndez, Meli\'{a}n \& Pestana \cite{FMP}, and computations of Hausdorff dimensions are found in Hill \& Velani \cite{HV95,HV97}, Stratmann \& Urba\'{n}ski \cite{StU}, Urba\'{n}ski \cite{Urb} and Reeve \cite{Rev}, for instance.

\smallskip

\item  {\bf Covering problems}:  In this case, given  $y_0\in X$, one is interested in the points which are  well approximated by the orbit of $y_0$, i.e. $$
C(T, \psi):=\Big\{x\in X: |T^ny_0-x|<\psi(n,x), \ {\text{i.o.}}\ n\in \N\Big\}.
$$
This is a dynamical version of the famous Dvoretzky covering problem \cite{Do}, see  Fan, Schmeling \& Troubetzkoy \cite{FST} and Liao \& Seuret \cite{LiS}  for the value of the Hausdorff dimension of  $C(T,\psi)$ for finite Markov maps $T$.

\smallskip
\item   {\bf Quantitative Poincar\'{e} recurrence properties}: Let $\psi: \N\times X\to \R^+$. One focuses on the points $x$ whose orbits  come back closer and closer  to  $x$  at a  rate $\psi$ possibly depending on $x$, i.e. the set
$$
   R(T,\psi):=\Big\{x\in X: |T^nx-x|<\psi(n, x), {\text{i.o.}}\ n\in \N\Big\}.
    $$
\end{itemize}

In this paper, we focus on the third question.
As far as the size in measure of $R(T,\psi)$ is concerned, Boshernitzan obtained the following  outstanding result for general systems.
\begin{theoremalph}[Boshernitzan \cite{Bo}]  Let $(X, T, \mu, d)$ be a measure dynamical system with a metric $d$. Assume that, for some $\alpha>0$, the $\alpha$-dimensional Hausdorff measure $\mathcal{H}^{\alpha}$ of the space $X$ is $\sigma$-finite. Then for $\mu$-almost all $x\in X$, \begin{equation}\label{1.5}
\liminf_{n\to \infty}{n^{\frac{1}{\alpha}}}d(T^nx, x)<\infty.
\end{equation}If, moreover, $\mathcal{H}^{\alpha}(X)=0$, then for $\mu$-almost all $x\in X$, $$
\liminf_{n\to \infty}{n^{\frac{1}{\alpha}}}d(T^nx, x)=0.
$$ \end{theoremalph}

Later, Barreira and  Saussol \cite{BaS} showed that the exponent $\alpha$ in (\ref{1.5}) is related to the lower local  dimension of $x$.  Tan and Wang \cite{TaW} considered the size of $R(T,\psi)$ in Hausdorff dimension when $(X,T)$ is the system of beta expansion.

 In this current work, we  consider  the quantitative Poincar\'{e} recurrence question in the setting of conformal iterated function systems. Before formulating our main result, let us recall the notation of conformal iterated function system (for a detailed  survey on infinite IFS, see the  works of Mauldin and Urba\'{n}ski \cite{MaU, MaU2,MaU1}).

\begin{dfn}
Let $(X,d)$ be a complete metric space.
Let $\Lambda$ be a countable index set with at least two elements and let $\S= \{\phi_i: \zu^d\to \zu^d, i\in \Lambda\}$ be a collection of injective contractions from  $\zu^d $
 into $\zu^d$.

The system $\S$ is supposed to be uniformly contractive, i.e.   there exists $0<\rho<1$ such that  for every $i\in \Lambda$ and for every pair of points $x, y\in X$,
\begin{equation}
\label{defrho}
d(\phi_i(x), \phi_i(y) )\le \rho \cdot  d(x,y).
\end{equation}
  Any such collection $\S$ of contractions is
called an iterated function system (denoted by IFS for brevity).
\end{dfn}

We are particularly interested in the properties
of the limit set defined by such a system. We can define this set as the image of the
coding space under a coding map as follows. Let $\Lambda^*=\bigcup_{n\ge 1}\Lambda^n$, the space of finite words, and $\Lambda^{\infty}=\Lambda^{\N}$ the
 collection of all infinite words with each letter in $\Lambda$. For $w\in \Lambda^n$, $n\ge 1$, let $\phi_w=\phi_{w_1}\circ  \phi_{w_2} \circ \ldots  \circ  \phi_{w_n}$.  If $w\in \Lambda^*\cup  \Lambda^{\N}$ and the integer $n\ge 1$ does not
exceed the length of $w$, we denote by $w|_n$ the word $(w_1 ,w_2, \ldots, w_n)$. Given $w\in \Lambda^{\infty}$,
since the diameters of the compact sets $\phi_{w|n}(X)$ $(n\ge 1)$ converge to zero, the set
$$\bigcap_{n\ge 1}\phi_{w|n}(X)$$
is a singleton and therefore, its element $\pi(w)$ defines the coding map $\pi : \Lambda^{\infty} \to X$. We call $w\in \Lambda^{\N}$ the code of $x$ if $\pi(w)=x$.
The main object in the IFS theory is the limit set defined as: $$
J=\pi(\Lambda^{\infty})=\bigcup_{w\in \Lambda^{\infty}} \bigcap_{n\ge 1}\phi_{w|n}(X)=\bigcap_{n\ge 1}\bigcup_{w: w\in \Lambda^n}\phi_w(X).
$$

Following the standard definitions (of Mauldin and Urba\'nski for instance), we introduce the open set condition and the property for an IFS to be conformal, which provides us a natural framework to work with.

\begin{dfn}  An IFS $\S=\{\phi_i: X\to X, i\in \Lambda\}$ is said to satisfy the open
set condition (OSC) when there exists a nonempty open set $U \subset X$ (in the topology of $X$) such
that
$$
\forall \, i\in \Lambda, \ \ \phi_i(U)\subset U, \ \ \mbox{ and } \ \  \phi_i(U)\cap \phi_j(U)=\emptyset \ \ \mbox{ whenever } {i\neq j}.$$

An IFS  $\S$   is
conformal if   the following conditions are satisfied:
\begin{enumerate}
\item
  $X \subset \R^d$ for some $d\ge 1$,
\item
 $\S$ satisfies the OSC with $U={\rm{Int}}_{\R^d}(X) $.
\item
 There exists an open connected set $X \subset V \subset \R^d$ such that all maps $\phi_i, i\in \Lambda$, extend to $C^1$ conformal diffeomorphisms of $V$ into $V$.
\item
There exist
$\gamma, l>0$ such that for every $x\in \partial X\subset \R^d$, there exists an open
cone $C_x $ with vertex $x$, central angle of Lebesgue measure $\gamma$, and altitude $l$, such that  $ C_x \subset{\rm{Int}}(X)$.

\item
 Bounded Distortion Property (BDP). There exists $K\ge 1$ such that  for every  points $x,y\in V$ and $w\in \Lambda^*$
\begin{equation}\label{ff2}\big|
\phi_w'(x)\big|\le K\big|\phi'_w(y)\big|,
\end{equation}
where $\phi'_w$ is the differential of $\phi_w$ and $|\phi'_w(x)|$ is the norm of 
$\phi'_w(x)$.
\end{enumerate}
\end{dfn}

\smallskip

  From now on, we work with $X=\zu^d$,  $d\geq 1$, endowed with the euclidian norm $|\cdot|$.

\smallskip

For infinite IFS, the limit set $J$ is not necessarily compact, and many points $x$ may have  multiple codings, i.e. there exist $w\neq w' \in \Lambda^{\mathbb{N}}$ such that $\pi(w)=\pi(w')=x$.

\begin{dfn}
We denote by $J^\ast \subset J$ the set of points $x\in J$ with unique coding.
\end{dfn}

We keep in mind that the set $J\setminus J^\ast$  shall be relatively small when compared to $J$ and $J^\ast$ when the IFS is conformal.

\medskip

A natural transformation $T:J^\ast\to J^\ast$ can be defined as follows. Without causing any confusion, for each $x\in J^\ast$, we write $$
x=[w_1,w_2,\ldots]  \ \ {\text{ when }} \  \ \pi(w)=x.
$$

For any $x\in J^\ast$,   define $$
T: x \in J^\ast \longmapsto Tx = T([w_1, w_2 ,w_3,\ldots]):= [w_2,w _3,\ldots].
$$

The transformation $T$ can just be viewed as the shift map in a subset of the coding space $\Lambda^{\N}$.

 It is clear that for any $x=(w_1, w_2, \ldots) \in J^\ast$, $\phi_{w_1}\circ T(x)=x$. So, we adopt the convention that the differential of $T$ is given by
 \begin{equation*}
T'(x)=\Big(\phi'_{w_1}(Tx)\Big)^{-1}, \ {\text{when}} \ x=[w_1,w _2,\ldots] .
\end{equation*}

In this paper, we consider the quantitative recurrence properties in the system $(J^\ast,T)$ generated by a conformal iterated function system $\S$.
Our aim is to study the set
\begin{equation*}
R(f):=\Big\{x\in J^\ast: |T^nx-x|<e^{-S_nf(x)} \ {\text{for infinitely many}}\ n\in \N\Big\},
\end{equation*}
where $S_nf(x)$ denotes the ergodic sum associated with  $f:\zu^d\to \R^+$  a positive function, defined by
$$S_nf(x) := f(x)+\ldots+f(T^{n-1}(x)).$$

We   study potentials $f$ satisfying the standard tempered distortion property.
\begin{dfn}
 Let $f:J\to \R$ be a function. The $n$-th variation of $f$, denoted by   ${\text{Var}}_n(f)$, is defined as
 $$
{\text{Var}}_n(f):=\sup_{w\in \Lambda^n: \,  x,y \, \in  \, I_n(w) }|f(x)-f(y)|.
$$
A function $f:J\to\R$ is said to fulfill the {\em tempered distortion property} if \begin{equation}\label{1}
{\text{${\rm{Var}}_1(f)<\infty$  \  \ and \ \   ${\rm{Var}}_n(f)\to 0$ as $n\to \infty$.}}
\end{equation}
\end{dfn}

We prove the following.

  \begin{thm}
  \label{t0} Let $\S$ be a conformal IFS, and  let   $f: \zu^d\to \R^+$ be a positive function.
 Assume that:
 \begin{enumerate}
 \item [(H1)]
  $f$ fulfills the  tempered distortion property (\ref{1}).
  \item [(H2)] Denoting
  \begin{equation}
  \label{defdim}
  s (f) = \inf\big\{s\ge 0: \P(T, -s(\log |T'|+f))\le 0\big\}, 
  \end{equation}
 where $\P$ is the pressure function associated with the IFS $\S$ (see Definition \ref{defpressure}), one has $\displaystyle s(f) >\dim_{\textsf{H}} (J\setminus J^\ast).$
  \end{enumerate}
  Then
$$
\dim_{\textsf{H}}R(f)=s (f) .$$
\end{thm}

Let us make some comments on our result:
\begin{itemize}
\smallskip\item  (H1) is a standard assumption on potentials.
\smallskip\item
Observe that (H2) implies $s(f)>0$. When $s(f)=0$, it is always the truth that $\dim_{\textsf{H}}R(f)=0$ (see Section \ref{section4} for the upper bound of the dimension of $R_2(f)$ which contains $R(f)$).
\smallskip\item
 The Hausdorff dimension of $R(f)$ is given in terms of the pressure function. This emphasizes our dynamical construction.
\smallskip\item
 Our assumption (H2) is mandatory in our approach since in our proof, we build a Cantor set sitting on $J$, not on $J^\ast$.  However, (H2) asserts that the points we are interested in (with quantitative recurrence properties) form a set with larger dimension  than the points with multiple codings, and thus the Cantor set sitting on $J$ is the one giving the right dimension to $R(f)$.
\smallskip\item
 Assumption (H2) is obviously verified when $d=1$, since in this case points with multiple codings in conformal IFS are known to be countable, or for IFS satisfying a strong open set condition. There are many other examples for which the dimension of $J\setminus J^\ast$ is  controlled (in terms of Hausdorff dimension), but of course, for   general IFS, (H2) may be difficult to check.
\end{itemize}
\medskip

Our article is organized as follows. Section \ref{section2} contains some preliminary results. In Sections \ref{section4} and \ref{section5}, we prove respectively the upper and the lower bound for the Hausdorff dimension of $R(f)$.   In Section \ref{section3} we give some applications of our results to some ``exotic" sets related to Diophantine approximations.

\section{Preliminaries}
\label{section2}

In this section, we define the cylinder set, present some well known results on the pressure function, and give a modification in defining $R(f)$.

For any $I\subset [0,1]^d$,  $|I|$  stands for the
diameter of $I$.

\subsection{Cylinders}

For each $(w_1,\ldots, w_n)\in \Lambda^n$,  we call
$$
I_n(w_1,\ldots,w_n)=\left\{x\in J:   \begin{cases} \ \exists \, w'=(w'_1,w'_2,...)\in \Lambda^\N \mbox{ such that } \pi(w') =x \\
 \  \mbox{ and }\forall\,   1\le i\le n, \ w'_i =w_i \end{cases}\right\}
$$
 a {\em cylinder of order $n$} or an {\em $n$th order cylinder}, which is the collection of the points in $J$ whose symbolic representations  begin  by $(w_1,\ldots,w_n)$.

For any $x\in J^\ast$, denote by $I_n(x)$ the $n$th order cylinder containing $x$.

Assume that $x\in J^\ast$,  so that $x =[w_1,w_2,\cdots]$ for some unique $w=(w_1,w_2,...)\in \Lambda^\N$. For each $n\ge 1$, set $\xi
= [w_{n+1},w_{n+2},...] \in J^\ast$. It is clear that $T^n\circ \phi_{(w_1,\cdots, w_n)}(\xi)=\xi$. This follows that $$
  |(T^n)'(x)|=|\phi'_{(w_1,\cdots,w_n)}(\xi)|^{-1}, \ {\text{with}}
\ x=\phi_{w_1,\cdots,w_n}(\xi).  $$
So by the bounded distortion property (\ref{ff2}), we have that for any $n\ge 1$ and $w\in \Lambda^n$,
\begin{equation}\label{f5*}
K^{-1}\le \left|\frac{|\phi'_{(w_1,\cdots,w_n)}(\xi_1)|^{-1}}{|\phi'_{(w_1,\cdots,w_n)}(\xi_2)|^{-1}}\right|\le K, \ \ {\text{for all}} \ \xi_1, \xi_2\in X,
\end{equation}
and
\begin{equation}\label{f5}
K^{-1}\le \left|\frac{(T^n)'(x_1)}{(T^n)'(x_2)}\right|\le K, \ \ {\text{for all}} \ x_1,x_2\in I_n(w)\cap J^\ast.
\end{equation}


We call (\ref{f5}) the bounded distortion property of $T$, which provides us with  a good control on the diameter of a cylinder.

Recall that for the conformal IFS $\S$, $K$ is the constant appearing in the bounded distortion property
  \eqref{ff2} and $\rho$ is the uniform bound  \eqref{defrho} for the contraction ratios of the mappings $(\phi_i)_{i\in \Lambda}$.
  \begin{pro}\label{p0}  For any $n\ge 1$ and $(w_1,\ldots, w_n)\in \Lambda^n$, the following  holds:

  \begin{enumerate}
  \item
  For any $x\in I_n(w_1,\ldots,w_n)$, the diameter of $I_n(w_1,\ldots,w_n)$ satisfies \begin{equation}\label{2.4}
K^{-1}|\phi'_{(w_1,\ldots,w_n)}(x)|^{-1} \le \big|I_n(w_1,\ldots, w_n)\big|\le K|\phi'_{(w_1,\ldots,w_n)}(x)|^{-1}.\end{equation}
\item For every $1 \leq k\leq n$,
 $$
K^{-1}\le \frac{|I_n(w_1,\ldots, w_n)|}{|I_{k}(w_1,\ldots,w_{k})| \cdot |I_{n-k}(w_{k+1},\ldots, w_n)|}\le K.
$$

\item  One always has $|I_n(w_1,\ldots, w_n)|\le \rho^{n}$.
\end{enumerate}
\end{pro}

\begin{rem}
We choose to take the same constant $K>1$ in all the bounded distortion-like inequalities, to facilitate the notations.
\end{rem}

We end with another remark about the definition of $T$ that will help us with our readability.
\begin{rem} When $x\in I_n(w)$, $T(x)$ is not uniquely defined, since a point $x$ may have multiple codings. But when there is no possible confusion, i.e. when we explicitly mention that $x\in I_n(w)$, we will denote by $T^n x$ the point $(\phi_w)^{-1}(x)$. This slight abuse of notation will ease our definitions. In particular,  when $x\in I_n(w)$,
\begin{itemize}
\item
  \eqref{f5*} and \eqref{f5} coincide.
\item
if $\psi:J\to \R$ is any function and any $x=[w_1,w_2,\cdots]\in J$, the Birkhoff sum $S_n \psi(x) $ means
$$S_n\psi(x) = \psi(x) + \psi\big ( (\phi_{w_1})^{-1}\big) (x)+\psi\big ( (\phi_{(w_1,w_2)})^{-1}\big) (x)+...+\psi\big ( (\phi_{(w_1,w_2,..., w_{n-1})})^{-1}\big) (x),$$
which coincides with
$$S_n\psi(x) = \psi(x)+\psi(Tx)+...+\psi(T^{n-1}x)$$
when $x\in J^\ast$.
\end{itemize}
\end{rem}
\subsection{Pressure function}


The topological pressure function $\P(T, \psi)$, with a potential $\psi$, for a conformal iterated function system is defined as follows:
\begin{equation}\label{ff1}\P(T, \psi)=\lim_{n\to\infty}\frac{1}{n}\log \sum_{w: w\in \Lambda^n}\sup_{x\in I_n(w)}\big\{e^{S_n\psi(x)}\big\}.\end{equation}
The existence of the limit follows from the sub-multiplicativity property:  for any $w\in \Lambda^n$ and $v\in \Lambda^m$, $$
\sup_{x\in I_{n+m}(w, v)}\big\{e^{S_{n+m}\psi(x)}\big\}\le \sup_{x\in I_n(w)}\big\{e^{S_n\psi(x)}\big\}\cdot \sup_{x\in I_m(v)}\big\{e^{S_m\psi(x)}\big\}.
$$

In the following, only the potential $\psi_s=-s(f+\log |T'|)$
 with $s\ge 0$ is concerned.
It is easy to check that this mapping $\psi_s$ satisfies the tempered distribution property  (\ref{1}). In this case, it is classical to see that the limit in (\ref{ff1}) is the same when the supremum over $x \in I_n(w) $ is replaced by $e^{S_n\psi(x)}$, for any choice of $x\in I_n(w)$.         Hence, in the sequel, when we need to  take a point $x$ in $I_n(w)=I_n(w_1,\ldots,w_n)$, we use the generic notation $x=[w]=[w_1,\ldots,w_n]$.  Finally,

\begin{dfn}
\label{defpressure}
The pressure function reduces to the following form:
\begin{equation}\label{f6}
 \P(T,\psi_s)=\lim_{n\to\infty}\frac{1}{n}\log \sum_{w\in \Lambda^n} \left(\big|(T^n)'([w])\big|^{-1}e^{-S_nf([w])}\right)^s,
 \end{equation}
where we use the abuse of notation $(T^n)'([w])$ to express $(\phi_w)^{-1}(x) $ for some point $x\in I_n(w)$.

\end{dfn}

\medskip

Let $A$ be a finite subset of $\Lambda$, and let
 $$
J_A=\Big\{x\in J: \exists\ w\in A^\N \mbox{ such that } \pi(w)=x\Big\}.
$$

Then $(J_A\cap J^\ast ,T)$ is a sub-system of $(J^\ast,T)$. We also define the pressure function restricted  naturally associated with $J_A$  as:
\begin{equation*}
\P_A(T, \psi_s)=    \lim_{n\to\infty}\frac{1}{n} \log \sum_{(w_1,\ldots,w_n)\in A^n} \left(\big|(T^n)'([w])\big|^{-1}e^{-S_nf([w])}\right)^s.
\end{equation*}

Applying the tempered distortion property of the potential $\psi_s$, we have the following continuity property of the pressure function.
\begin{pro}[\cite{MaU}]\label{p1}
One has

\begin{enumerate}
\item Let  $(\psi_s^n)_{n\geq 1}$ be a sequence of functions converging to $\psi_s$ in the supremum norm. Then $\lim_{n\to\infty}\P(T,\psi_s^n)=\P(T, \psi_s)$  

\item
One has
$$
\P(T, \psi_s)=\sup\Big\{\P_A(T, \psi_s): A\ {\text{is a finite subset of}}\ \Lambda \Big\}.$$
\end{enumerate}
\end{pro}

\subsection{Refinement on $R(f)$ and simplification of the problem}
\label{simplification}

In this short section, we will give some modifications on $R(f)$ at first, and then explain that Theorem \ref{t0} can be deduced by restricting $\S$ being only a {\em finite} conformal iterated function system.

Let us introduce the sets
$$
R_1(f):=\Big\{x\in J : |T^nx-x|<e^{-S_nf(x)} \ {\text{for infinitely many}}\ n\in \N\Big\},
$$
and
$$
 {R}_2(f)= \ \bigcap_{N=1}^{\infty}  \  \bigcup_{n=N}^{\infty} \ \bigcup_{(w_1,\ldots,w_n)\in \Lambda^n} \Big\{x\in I_n(w_1,\ldots,w_n): |T^nx-x|<e^{-S_nf([w_1,\ldots,w_n])}\Big\}.
$$
Recall that in the two preceding formulae, $T^n x$ stands for $(\phi_w)^{-1}(x)$ when $x\notin J^\ast$. In particular, the sets $R_1(f)$ and $R_2(f)$ are included in $J$ (not in $J^\ast$), and the set we are really interested in, $R(f)$, is included in $R_1(f)$. More precisely, it consists exactly in the points belonging to $R_1(f)$ and $J^\ast$ simultaneously.

\medskip

In the definition of $R(f)$ and $R_1(f)$, the shrinking speed $e^{-S_nf(x)}$ depends on $x$. This makes the things a little uneasy to manage since $x$ has not been determined yet. We relax the dependence of the shrinking speed on $x$ as follows. Fix $\ep>0$. By the tempered distortion property of $f$, there exists an integer $N_0=N(\ep, f)$ such that for any $n\ge N_0$,  \begin{equation}\label{b1}
\big|S_nf(x)-S_nf(y)\big|<n\ep, \ \ \forall \ w\in \Lambda^n, \ x,y\in I_n(w).
\end{equation}

It follows from the inequality  (\ref{b1}) that one has the successive embedding
\begin{equation}
\label{embed}
\mbox{for every $\ep>0$, \ \ }  {R}_2(f+\ep)\subset R_1(f)\subset {R_2}(f-\ep).
\end{equation}
Applying the continuity of the pressure function (Proposition \ref{p1} (1)),  in order to prove that the Hausdorff
dimension of $R_1(f)$ is equal to  $s(f)$ (defined by formula \eqref{defdim}),  it suffices to show that
\begin{equation}\label{f7}
\dim_{\textsf{H}} {R}_2(f)=s(f).
\end{equation}

\medskip

Let us explain why this  is  enough to get Theorem  \ref{t0}. Assume that    \eqref{f7} is proved. Obviously, this property combined with \eqref{embed}    imply that $  \dim_{\textsf{H}} R_1(f) =s(f)$.   By our assumption (H2), which states that  the Hausdorff dimension of $J\setminus J^\ast$ (the set of points with multiple codings) is strictly less than $s(f)$, it follows that necessarily
$$\dim_{\textsf{H}}( (J\setminus J^\ast ) \cap  {R}_1(f)) <  \dim_{\textsf{H}}  (J\setminus J^\ast ) = s(f).$$
This yields that
$$\dim_{\textsf{H}}(   J^\ast   \cap  {R} _1(f))  = s(f).$$
Recalling that $ J^\ast   \cap  {R} _1(f) = R(f)$, this concludes  the proof of Theorem   \ref{t0}.

\medskip

A last simplification consists in applying the second item    of Proposition \ref{p1}, which authorizes us to  restrict ourselves to finite confomal systems. Finally,  Theorem \ref{t0}  will follow if we can show the following result.

 \begin{thm}\label{t2}Let $\S$ be a finite conformal IFS on $[0,1]^d$. Assume that $f$ has tempered distortion property. The Hausdorff dimension of $ {R}_2(f)$ is  $s(f) =\inf\{s\geq 0: \P(T, -s(\log |T'|+f))=0\}.$
\end{thm}

\section{Upper bound for the Hausdorff dimension of $R_2(f)$}
\label{section4}

Recall that now $\Lambda$ is supposed to be a finite set of indices.

\medskip

The argument on the upper bound of $\dim_{\textsf{H}}{R_2}(f)$ is quite standard by using its natural covering systems. Recall that
\begin{equation}
\label{defr2}
{R_2}(f)=\bigcap_{N=1}^{\infty} \ \bigcup_{n=N}^{\infty}\ \bigcup_{(w_1,\ldots,w_n)\in \Lambda^n} J_n(w_1,\ldots,w_n),
\end{equation}
 where
\begin{equation}
\label{defjn}
J_n(w_1,\ldots,w_n)=\Big\{x\in I_n(w_1,\ldots,w_n): \big|T^nx-x\big|<e^{-S_nf([w_1,\ldots,w_n])}\Big\}.
\end{equation}
One needs to keep in mind that $T^n x$ stands for $(\phi_{(w_1,\ldots,w_n)})^{-1}(x)$ when $x\in  I_n(w_1,\ldots,w_n)$, even when $x\notin J^\ast$.
Then for each $N\ge 1$, the collection of sets $$
\Big\{J_n(w_1,\ldots,w_n): (w_1,\ldots,w_n)\in \Lambda^n, n\ge N\Big\}
$$ is a natural covering system of ${R_2}(f)$. Now we estimate the diameter of $J_n(w_1,\ldots,w_n)$ for any $(w_1,\ldots,w_n)\in \Lambda^n$.

\smallskip

\begin{lem}\label{l3}
For any $n\ge 1$ large enough  and $w=(w_1,\ldots,w_n)\in \Lambda^n$, $$
\big|J_n(w_1,\ldots,w_n)\big|\le 2K\big|(T^n)'([w_1,\ldots, w_n])\big|^{-1}\cdot e^{-S_nf([w_1,\ldots,w_n])}.
$$
\end{lem}

\begin{proof}
  Fix $n \geq 3$ and $ w=(w_1,\ldots,w_n)\in \Lambda^n$, and consider  $y=[w^{\infty}]$   the infinite word with a periodic symbol representation of period $w$. By construction, $y\in I_n(w)$ and  $T^n(y)=y$, hence  $y\in J_n(w)$. For any other point  $x\in J_n(w)$, the triangle inequality gives that
$$
|T^nx-T^ny|\le |T^nx-x|+|x-y|+|y-T^ny|\le e^{-S_nf([w_1,\ldots,w_n])}+ |x-y|.
$$
Here, $T^n $  means $(\phi_{w})^{-1}$, which is an expansive mapping (essentially due to  \eqref{defrho}). If  $\tilde x = T^nx$ and $\tilde y=T^n y$, one has
$$|x-y|= |\phi_w( \tilde x) - \phi_w(\tilde y)| \leq          \sup_{z\in I_n(w)}  |\phi'_w(z)| \cdot |\tilde x- \tilde y|=     \sup_{z\in I_n(w)}  |(T^n)'(z)|^{-1} \cdot|T^nx-T^ny|  .$$
From the last two inequalities, we deduce that
$$
|x-y|<\left(  \left(\sup _{z\in I_n(w)}   |(T^n)'(z)| ^{-1}  \right)^{-1}  -1\right)^{-1}\cdot e^{-S_nf([w_1,\ldots,w_n])}.
$$
 The term $\sup _{z\in I_n(w)}  |(T^n)'(z)|^{-1}$ is large since $\phi_w^{-1}$ is expansive, thus for $n$ large,
 $$
|x-y|< 2 \sup _{z\in I_n(w)}  \big|(T^n)'(z)  \big|^{-1} \cdot e^{-S_nf([w_1,\ldots,w_n])},
$$
and   the result follows by the bounded distortion property  (\ref{f5*}) and \eqref{f5}.
\end{proof}

\medskip

We conclude now regarding the upper bound for the Hausdorff dimension of $R_2(f)$.
By Lemma \ref{l3}, and using
 $$
\Big\{J_n(w_1,\ldots,w_n): (w_1,\ldots,w_n)\in \Lambda^n, n\ge N\Big\}
$$
as covering of $R_2(f)$, the $s$-dimensional Hausdorff measure $\mathcal{H}^s$ of ${R_2}(f)$ can be estimated as \begin{align*}
\mathcal{H}^s({R_2}(f))&\le      \liminf_{N\to \infty}\sum_{n=N}^{\infty}\sum_{w\in \Lambda^n} \big|J_n(w)\big|^s\\
&\le (2K)^s \liminf_{N\to \infty}\sum_{n=N}^{\infty}\sum_{w\in \Lambda^n} \left(\big|(T^n)'([w])\big|^{-1}\cdot e^{-S_nf([w])}\right)^s.
\end{align*}
Thus by the definitions of the pressure function $\P$ (\ref{f6}) and $s(f)$ (\ref{f7}), the above estimation yields that for every fixed $s> s(f)$,
the $s$-dimensional Hausdorff  measure $\mathcal{H}^s(R_2(f))$ is zero. This gives that
$\dim_{\textsf{H}}{R_2}(f)\le s$. Since this holds true for every $s>s(f)$, the conclusion follows.

\section{Lower bound for the Hausdorff dimension of $R_2(f)$}

\label{section5}

\subsection{Preliminaries}

Let us introduce
 \begin{equation}
 \label{defeta}
\eta:=\min\{|\phi'_i(x)|: x\in J, \ i\in \Lambda\}  \ \ \mbox{ and } \ \
\eta_m=\min\Big\{|I_m(w)|: w\in \Lambda^m\Big\},
\end{equation}
and recall that
$$\rho:=\max\{|\phi'_i(x)|: x\in J,\  i\in \Lambda\}<1.
$$
Since $\S$ is a finite conformal iterated function system,  by the bounded distortion property (\ref{ff2}), we have $\eta>0$ (and thus $\eta_m>0$).
Besides Proposition \ref{p0} on the diameter of a cylinder,  we also have that
\begin{equation}\label{f4}
\eta\le \frac{|I_n(w_1,\ldots, w_n)|}{|I_{n-1}(w_1,\ldots,w_{n-1})|}\le 1
\end{equation}
 for any $(w_1,\ldots, w_n)\in \Lambda^n$ and $n\ge 1$. Moreover, since $\S$ is finite and $f$ satisfies the tempered distortion, we have $$
\|f\|_{\infty}=\sup\{|f(x)|: x\in J\}<\infty.
$$

Now we define some numbers which are closely connected with the dimension of ${R_2}(f)$.

By the definitions of the pressure function $\P$ and $s(f)$, the following is a standard result.
\begin{pro}\label{p3} For each $n\ge 1$, define $s_n(f)$ as the unique solution to
\begin{equation}
\label{defsn}
\sum_{(w_1,\ldots,w_n)\in \Lambda^n}\left(\big|(T^n)'([w])\big|^{-1}\cdot e^{-S_nf([w])}\right)^s= 1.
\end{equation}
Then $\lim_{n\to\infty}s_{n}(f)=s(f).$
\end{pro}

The lower bound of the Hausdorff dimension of ${R_2}(f)$ will be estimated by using the
classical mass distribution principle.
For this purpose, we are going to construct a   Cantor set $\F_{\infty}$ inside ${R_2}(f)$ and simultaneously a probability measure $\mu$ supported on $\F_{\infty}$, with the correct scaling behavior. More precisely we are going to show that, for every $x\in \F_{\infty}$, the lower local   dimension of $\mu$ at $x$
 satisfies
\begin{equation*}
\liminf_{r\to 0}\frac{\log \mu(B(x,r))}{\log r}\ge s(f).
\end{equation*}
Then by the mass distribution principle \cite[Proposition 4.2]{Fal}, we conclude that
\[
\dim_{\textsf{H}}\F_{\infty}\ge s(f).
\]
This yields Theorem \ref{t2}, and as we explained in Section \ref{simplification}, this also finishes the proof of Theorem \ref{t0}.

\medskip

A general idea to find points in ${R_2}(f)$ (defined by \eqref{defr2} and \eqref{defjn}) is outlined in the following fact which was also used in \cite{TaW} in order to study the quantitative recurrence properties in beta expansions:

 {\em Two points $x$ and $y$  are close when their symbolic representations share a common prefix for a long run. As  far as the  points  $x$ and $y=T^nx$ are concerned, they are close enough when there is a
repetitive prefix in the symbolic representation of $x$.}
\medskip

Utilizing the above idea, we present a way to realize the event $J_{n}(w_1,\ldots,w_n)$ rigorously.
\begin{lem}\label{l2}Let $w\in \Lambda^*$ be a finite word of length $n$ and $r>0$. Write $\overline{w}=w^{\infty}$ the infinite periodic word with periodic pattern $w$. Consider the unique integer $t$ such that $$
|I_t(\overline{w})|< r\le |I_{t-1}(\overline{w})|.
$$ Then for any $x\in I_{n+t}(\overline{w})$, one has $$
\big|T^nx-x\big|\le |I_t(\overline{w})|<r.
$$
Let   $w^*$  be the word $\overline{w}|_t$. Then,
\begin{align}
 |I_{n+t}(w w^*)|\geq  K^{-1} \eta r |I_n(w)|. \label{ff6}
\end{align}

\end{lem}

Lemma \ref{l2} will be used along the proof of Theorem \ref{t2}.

\begin{proof}
 Again, notice that $T^n$ means $\phi_w^{-1}$ when $x\in I_{n+t}(\overline{w})$. For any $x\in
I_{n+t}(\overline{w})$,
 by the periodicity of $\overline{w}$, we have $
 T^nx\in I_t(\overline{w}).
 $ This means that both $x$ and $T^nx$ are in the same cylinder $I_t(\overline{w})$. Trivially,
$$\big|T^nx-x\big|\le |I_t(\overline{w})|<r.
$$

 If $w^*=\overline{w}|_t$, by construction, one has
$$
|T^nx-x|<r, \ {\text{for all}} \ x\in I_{n+t}(w  w^*),
$$ and by Proposition \ref{p0}   and (\ref{f4}) that
$$ |I_{n+t}(w w^*)|\ge K^{-1} |I_n(w)| \cdot |I_t(w^*)|\ge K^{-1}|I_n(w)|\cdot \eta r.
$$

\end{proof}

We will also use repeatedly the following lemma.

\begin{lem}
\label{l4}
There exists a constant $\tilde\eta>0$ (depending on the dimension $d$ only) such that for every large  integer $m$,    for every finite word $v\in \Lambda^*$ of length $t \geq 0$, there exists a subfamily $\Gamma(v)$ of $\Lambda^m$ such that
\begin{equation}
\label{minimi}
\sum_{w\in \Gamma(v)}\left(\big|(T^m)'([w])\big|^{-1} \cdot e^{-S_mf([w])}\right)^{s_m(f)}\ge \tilde\eta,
\end{equation}
 and additionally,  for any $w,w'\in \Gamma(v)$ with $w\ne w'$, the distance between $I_{t+m}(vw)$ and $I_{t+m}(v w')$ is larger than $\eta_m|I_{t}(v)|$.
 \end{lem}

\begin{proof}
Recall the definitions  \eqref{defeta} and \eqref{defsn} of $\eta_m$ and $s_m(f)$.

Let us recall a geometric consequence of bounded distortion property:
\begin{equation}\label{7}
 \phi_w\Big(B(x,r)\Big)\supset B\Big(\phi_w(x), K^{-1}|\phi'_w| r\Big)
\end{equation}
for every $x\in X$, every $0<r\le {\text{dist}}(X, \partial U)$ and every word $w\in \Lambda^*$.

Let $I_t(v)$ be a cylinder of order $t$ and let $$
\mathcal{S}(v)=\Big\{I_{t+m}(v, w): w\in \Lambda^m\Big\}.
$$
We first prove that  there exists a partition of $\mathcal{S}(v)$ by at most  $(33Kd)^d$ families inside each of which  the members are apart from each other with a distance at least $\eta_m |I_t(v)|$.  The proof uses the same ideas as the Besicovitch's covering lemma \cite{mattila}.

 \medskip

 Assume that we are given as many as possible ``baskets" $\{{\mathcal{F}}_i\}_{i\ge 1}$. We will put the elements in $\mathcal{S}$ into 
 these baskets one by one in the following way. At first we arrange the elements in $\mathcal{S}$ according to their diameters in a decreasing order, denoted by $L_1,L_2,\ldots$. The process begins as follows.

Put $L_1$ to $\mathcal{F}_1$. If $L_2$ lies apart from $L_1$ with a distance larger than $\eta_m|I_t(v)|$, put $L_2$ into $\mathcal{F}_1$, otherwise put it into $\mathcal{F}_2$. Assume that $L_1, L_2,\ldots, L_k$ have already been put into a finite number of  baskets denoted by  $\mathcal{F}_1,\ldots, \mathcal{F}_{i_0}$.  If exist, choose the smallest integer $1\leq i\leq i_0$ such that  $L_{k+1}$  lies apart from every element in   $\mathcal{F}_{i}$ with a distance larger than $\eta_m|I_t(v)|$, and  then put $L_{k+1}$ into $\mathcal{F}_{i}$. If such an integer $i$ does not exist, put $L_{k+1}$ into the new basket $\mathcal{F}_{i_0+1}$. In this way, we give a partition of $\mathcal{S}$.

Further we prove that  at most $\kappa=(33Kd)^d$ many baskets are used. Assume on the contrary that the baskets $\mathcal{F}_1,\ldots,\mathcal{F}_{\kappa}, \mathcal{F}_{\kappa+1}$ are all nonempty. Let $L$ be the first element put into $\mathcal{F}_{\kappa+1}$. By the process above, at this moment, the first $\kappa$ baskets $\{\mathcal{F}_i, 1\le i\le \kappa\}$ are all nonempty and the diameters of the elements in these baskets are all greater than that of $L$. Moveover, since $L$ is put into a new basket $\mathcal{F}_{\kappa+1}$, for each $1\le i\le \kappa$, there exists an element denoted by $I_{t+m}(v, w_i)\in \mathcal{F}_i$ lying within a distance less than $\eta_m|I_t(v)|$ from $L$.

Write $r=|L|\ge \eta_m |I_t(v)|$ and $L\subset B(x_0,r)$ for some $x_0\in X$. Then $$
I_{t+m}(v, w_i)\cap B(x_0,2r)\ne \emptyset, \ \ {\text{for all}}\ 1\le i\le \kappa.
$$
Fix one integer $1\le i\le \kappa$. We will construct a ball inside $I_{t+m}(v, w_i)$ lying close to the ball $B(x_0,2r)$.

Let $r_1=1/4 |\phi'_{v,w_i}|^{-1} r.$ Since $r=|L|\le |I_{t+m}(v,w_i)|$, $r_1\le 1/4$. Let $y_0\in [0,1]^d$ such that $$
\phi_{v,w_i}(y_0)\in I_{t+m}(v, w_i)\cap B(x_0,2r).
$$
Since $r_1$ is small, there is enough room for us to find a cube $C_i$ in $(0,1)^d$ with sidelength $r_1$ and one vertex $y_1$ lying within a distance $\le r_1$ from $y_0$. On one hand, by (\ref{7}), the set $\phi_{v,w_i}(C_i)$ contains a ball $B_i$ with radius $\ge \frac{1}{8K}r$ inside $\phi_{v,w_1}((0,1)^d)$. On the other hand, this ball $B_i$ lies close to $B(x_0,2r)$ in the following sense: for any $y\in C_i$,
\begin{eqnarray*}
|\phi_{v,w_i}(y)-x_0|
&\le& |\phi_{v,w_i}(y)-\phi_{v,w_i}(y_1)|+|\phi_{v,w_i}(y_1)-\phi_{v,w_i}(y_0)|+|\phi_{v,w_i}(y_0)-x_0|\\
&\le & \sqrt{d} |\phi'_{v,w_i}| r_1+ |\phi'_{v,w_i}| r_1+2r
\\ &= & \sqrt{d} r/4+ r/4+2r\le 4dr.
\end{eqnarray*}
As a result, $B_i\subset B(x_0, 4dr).$

Finally, for each $1\le i\le \kappa$, we have constructed a ball $B_i$ with radius $\ge \frac{1}{8K}r$ inside $\phi_{v,w_1}((0,1)^d)$ and contained in $B(x_0,4dr)$. By the open set condition, all these balls $\{B_i: 1\le i\le \kappa\}$ are disjoint. Thus, a simple volume argument yields that
\begin{align*}
\kappa\cdot (\frac{1}{8K}r)^d {\rm{Vol}}(B(0,1))\le  \sum_{i=1}^{\kappa}{\rm{Vol}}(B_i)\le {\rm{Vol}}(B(x_0,4dr))\le (4dr)^d {\rm{Vol}}(B(0,1)).
\end{align*} This follows that $\kappa \le (32d K)^d$, a contradiction.

\medskip

The latter proves that  the cylinders $\mathcal{S}(v)$ can by divided into at most $(33Kd)^d$ families $(\mathcal{F}_i)$ of pairwise disjoint balls.
Hence, at least for one family $ \mathcal{F}_i$, one necessarily has
\begin{align*}
\sum_{  w\in \Lambda^m: I_{t+m}(v,w)\in \mathcal{F}_i}   & \left(\big|(T^m)'([w])\big|^{-1} \!\cdot\! e^{-S_mf([w])}\right)^{s_m(f)} \\
& \ \ \ \ \ \ \geq    \frac{1}{(33Kd)^d} \sum_{w\in \Lambda^m}\left(\big|(T^m)'([w])\big|^{-1} \!\cdot \! e^{-S_mf([w])}\right)^{s_m(f)}
\geq   \frac{1}{(33Kd)^d},
\end{align*}
where the definition \eqref{defsn} of $s_m(f)$ has been used.  This proves the lemma by choosing $\tilde \eta =\frac{1}{(33Kd)^d}$.
\end{proof}

\subsection{The Cantor subset}

We define a Cantor subset $\F_{\infty}$  of $R_2(f)$ level by level.
Recall  that necessarily $s(f)>0$ due to our assumption (H2). From another point of view, there is nothing need to be proven when $s(f)=0$, since the dimension of $R_2(f)$ is always bounded from above by $s(f)$ (see Section \ref{section4}).
\medskip

  Now, fix a small positive number $$0< \ep<  \frac{1}{2}\min( s(f), - \log\rho ),$$
  where $\rho$ is given by \eqref{defrho}. Then choose a large integer $m$ such that for any $n\ge m$, the following four conditions are fulfilled:
\begin{itemize}
\item  for every  word $w$ of length $n$ and every $x, y\in I_n(w)$,
\begin{equation}\label{ff4}
|S_nf(x)-S_nf(y)|\le \ep n,\end{equation}
followed by the tempered distortion property of $f$.

\item
one has
\begin{equation}\label{ff3}
|s_m(f)-s(f)|<    \ep.
\end{equation}
followed by Proposition \ref{p3}.

\item for every $w\in \Lambda^m$,
\begin{equation}
\label{ff3bis}
\left(\big|(T^m_w)'([w])\big|^{-1} \cdot e^{-S_mf([w])}\right)^{\frac{4\ep}{-\log \rho}} \le  K^{-2} e^{-3m\ep}.
\end{equation}

\item $m$ is so large that
\begin{equation}
\label{ff10}
e^{m\ep} \geq K^2.
\end{equation}

\end{itemize}

 Observe that the third inequality  can be realized since $\big|(T^m_w)'([w])\big|^{-1} \leq \rho^m$ and $f$ is positive. More precisely,
 $$\left(\big|(T^m_w)'([w])\big|^{-1} \cdot e^{-S_mf([w])}\right)^{\frac{4\ep}{-\log \rho}}  \leq (\rho^m)^{\frac{4\ep}{-\log \rho}} \leq  e^{-4m\ep} .$$

 \medskip

From now on the integer $m$ is fixed ensuring that (\ref{ff4}), (\ref{ff3}), \eqref{ff3bis} and (\ref{ff10}) are all fulfilled.

\subsubsection{Level 1 of the Cantor set}
 Let $n_0 \geq 1$ be an integer such that $ (2K)^{-1}\eta^{-n_0}\ge 2$, $t_0=1$ and let $m_1$ be a multiple of $m$ such that
$$
m_1 \ge n_{0}+t_{0} \ \ \mbox{ and } \ \  n_{0}+t_{0}\|f\|_{\infty}  \le  m_1  \ep.
$$
We write
$$m_1=\ell_1 m.$$

 Applying  Lemma \ref{l4} to $v=\emptyset$ gives us    a subfamily $\Gamma(\emptyset)\subset \Lambda^m$ of words such that the cylinders $\{I_m(w_1), w_1\in \Gamma(\emptyset)\}$ are far away from each other, and satisfy \eqref{minimi}.
Then, we define a sequence of sub-cylinders of $v=\emptyset$: First let $\Gamma_1=\Gamma(\emptyset)$ and define
$$
\F_1^{(1)} =\Big\{I_{m}( w_1): w_1\in \Gamma_1\Big\}.
$$

Assume that for every  $1\leq i \leq \ell $,  a finite set of words  $\Gamma_i $ of  length $m$ has been constructed, and  that  the collection  of cylinders of order $\ell m$
$$
\F_1^{(\ell )} =\Big\{I_{ \ell m}( w_1, \ldots, w_{\ell }):  \mbox{ for every $1\leq i\leq\ell $, } w_i\in \Gamma_ i\Big\}
$$
 has been defined.
Then for each word $v=(w_1,\ldots,w_\ell)\in \prod_{i=1}^{\ell}\Gamma_i$ of length $\ell m$, applying Lemma \ref{l4} to $v $ gives us  a family $\Gamma_{\ell+1}$ of words of length $m$ (i.e. $\Gamma_{\ell+1} \subset  \Lambda^m$). Then one sets
\begin{align*}
\F_1^{(\ell+1)}  =\Big\{I_{ (\ell+1)m}(  w_1, \ldots, w_{\ell}, w_{\ell+1}):  \mbox{ for every $1\leq i\leq \ell+1$, } w_i\in \Gamma_ i\Big\}.\end{align*}
Iterating this procedure until $\ell=\ell_1$, we obtain the collection of cylinders of order $\ell_1 m=m_1$ as \begin{align*}
\F_1^{(\ell_1)}  =\Big\{I_{ m_1}(  w_1, \ldots, w_{\ell_1}):  \mbox{ for every $1\leq i\leq \ell_1$, } w_i\in \Gamma_ i\Big\}.
\end{align*}

\medskip

Further, for each $(w_1,\ldots,w_{\ell_1})$ with $w_i\in \Gamma_i$ ($1\le i\le \ell_1$), we apply Lemma \ref{l2} to the word
 $$w^{(1)}=( w_1,\ldots, w_{\ell_1}  ) \ \ {\text{and}}\  r=e^{-S_{m_1}f([w^{(1)}])},$$
 to get  a word  $w^{(1, *)}$ of length $t_1$, satisfying the conditions of Lemma \ref{l2}, i.e.
\begin{eqnarray*}
 I_{m_1+t_1}(w^{(1)}, w^{(1,*)})&\subset  & J_{m_1}(w^{(1)})\nonumber,\\
\big|I_{m_1+t_1}(w^{(1)}, w^{(1,*)})\big|&\ge &   K^{-1} \eta\big|I_{m_1}(w^{(1)})\big|\cdot e^{-S_{m_1}f([w^{(1)}])}.
\end{eqnarray*}
Recall the definition \eqref{defjn} of $J_n(w)$.


Finally, let us set $n_1=m_1+t_1$ and define  $$
\F_{1} =\Big\{I_{n_1}(  w^{(1)}, w^{(1, *)}):  w^{(1)}=(w_1,\cdots, w_{\ell_1}) \mbox{ and    $\forall \,1\leq j\leq \ell_1$, } w_j\in \Gamma_ j\Big\},
$$ a collection of cylinders, and
 $$
\G_{1} =\Big\{(  w^{(1)}, w^{(1, *)}):  w^{(1)}=(w_1,\cdots, w_{\ell_1}) \mbox{ and    $\forall \,1\leq j\leq \ell_1$, } w_j\in \Gamma_ j\Big\},
$$ a collection of words corresponding to the cylinders in $\F_1$.
Both of them are called the first level of the Cantor set if no confusions arise.

\begin{rem} Note that the family $\Gamma_i$ depends on the previous families $\Gamma_1,\cdots,\Gamma_{i-1}$.
For simplicity, we don't emphasis this dependence in notation.

\end{rem}

\begin{rem}  Pay attention to the fact that the length of $w^{(1,*)}$ (i.e. the integer  $t_1$) depends on $w^{(1)}$ in Lemma \ref{l2}. Thus for different words $w^{(1)} $, the integer $n_1$ may be different. We omit the dependence in the notation for clarity, but one needs to keep that property in mind. Nevertheless, there is a uniform upper bound for these integers $t_1$. More precisely, by the choice of $t_1$ in Lemma \ref{l2} we have $$
e^{-m_1 \|f\|_{\infty} }\le e^{-S_{m_1  }f([w^{(1)}])}\le |I_{t_1-1}\big(({w}^{(1)})^{\infty}\big)|\le K \rho^{t_1-1}.
$$This yields $$
t_1\le - m_1  \|f\|_\infty \log \rho +1.
$$
\end{rem}

\subsubsection{$k$-th level of the Cantor set}

Assume that the $(k-1)$th level $\F_{k-1}$, a collection of cylinders, and simultaneously $\G_{k-1}$, a collection of words corresponding 
to the cylinders in $\F_{k-1}$,  of the Cantor set have been constructed.

We can choose $m_k$ such that  $m_k$ is a multiple of $m$ so large that if
 \begin{equation}\label{ff9bis}
\widetilde n_{k-1} := \max\{n_{k-1} : n_{k-1} \mbox{ is associated with the order of a cylinder in} \ \F_{k-1}  \},
\end{equation}
then
\begin{equation}
\label{ff9}
m_k/k \ge \widetilde  n_{k-1}  \ \ \mbox{ and } \ \  \ \widetilde n_{k-1}(1+\|f\|_\infty)\le m_k \ep.
\end{equation}
We write
$$m_k = m \ell_k.$$

Let $\ep^{(k-1)}$ be a word in $\G_{k-1}$, and    $I_{n_{k-1} }(\ep^{(k-1)}) $ the corresponding cylinder in $\F_{k-1}$.

\medskip

We start by applying Lemma \ref{l4} to the word $v=\ep^{(k-1)}$ to  get a subfamily $\Gamma_1(\ep^{(k-1)})\subset \Lambda^m$ satisfying the conditions of that lemma. Further, one sets
$$
\F_k^{(1)}\Big( \ep^{(k-1)}\Big)=\Big\{I_{n_{k-1}+ m}(\ep^{(k-1)}, w_1): w_1\in \Gamma_1 (\ep^{(k-1)}) \Big\}.
$$
As we did for the first step of the construction of the Cantor set, we assume that the  finite sets of words  $(\Gamma_i(\ep^{(k-1)}))_{i=1,...,\ell}$ of  length $m$ have been constructed, and  that  the collection  of cylinders of order $n_{k-1}+\ell m$
$$
\F_k^{(\ell )}  \Big( \ep^{(k-1)}\Big) =\Big\{I_{ n_{k-1}+ \ell m}( \ep^{(k-1)},w_1, \ldots, w_{\ell }):  \mbox{$\forall\, 1\leq i\leq\ell $, } w_i\in \Gamma_ i (\ep^{(k-1)}) \Big\}
$$
 has been defined.
Then for each word $v=(\ep^{(k-1)},w_1,\ldots,w_\ell)$ with $(w_1,\ldots, w_{\ell})\in \prod_{i=1}^{\ell}\Gamma_i(\ep^{(k-1)})$
of length $n_{k-1}+\ell m$, applying Lemma \ref{l4} to $v $ gives us  a family $\Gamma_{\ell+1}(\ep^{(k-1)})$ of words of length $m$, and one defines
\begin{eqnarray*}
 \F_k^{(\ell+1)} \big(\ep^{(k-1)}\big) = \Big\{I_{ n_{k-1}+(\ell+1)m}( \ep^{(k-1)}, w_1, \ldots,  w_{\ell+1}):  \mbox{ $\forall\, 1\leq i\leq \ell+1$, } w_i\in \Gamma_ i(\ep^{(k-1)})\Big\}.
\end{eqnarray*}
Iterating this procedure until $\ell=\ell_k$, we obtain the collection of cylinders of order $n_{k-1}+m_k$
\begin{eqnarray*}
 \F_k^{(\ell_k)} \big(\ep^{(k-1)}\big)  =\Big\{I_{ n_{k-1}+m_k}( \ep^{(k-1)}, w_1, \ldots, w_{\ell_k}   ):  \mbox{ $\forall\, 1\leq i\leq \ell_k$, } w_i\in \Gamma_ i(\ep^{(k-1)})\Big\}.
\end{eqnarray*}

%

\begin{rem}\label{remark}
\label{remdist}
Observe that by Lemma \ref{l4} any two different cylinders $$I_{  n_{k-1}+\ell m}( \ep^{(k-1)}, w_1,\ldots, w_{\ell-1}, w_{\ell}) \ {\text{and}} \ I_{  n_{k-1}+\ell m}( \ep^{(k-1)}, w_1,\ldots, w_{\ell-1}, w'_{\ell})$$ in $ \F_k^{(\ell)} \big(\ep^{(k-1)}\big)  $ are separated by a distance at least $\eta_m  \big| I_{  n_{k-1}+m(\ell-1)}( \ep^{(k-1)}, w_1,\ldots, w_{\ell-1})  \big|$.
\end{rem}

 Next, for every word $\ep^{(k-1)}$ and every $(w_1,\ldots,w_{\ell_k})$ with $w_i\in \Gamma_i(\ep^{(k-1)})$ ($1\le i\le \ell_k$), applying Lemma \ref{l2} to
\begin{equation}
\label{defwk}
w^{(k)}=(\ep^{(k-1)}, w_1,\ldots, w_{\ell_k}) \ \ \  {\text{and}} \ \ \  r=e^{-S_{n_{k-1}+m_k}f([w^{(k)}])}
\end{equation}
 gives us a word   $w^{(k, *)}$ of length $t_k$ satisfying the conditions of Lemma \ref{l2}, i.e.
\begin{eqnarray}
 I_{  n_{k-1}+m_k +t_k}(w^{(k)}, w^{(k,*)})&\subset  & J_{n_{k-1}+m_k}(w^{(k)})\label{b5},\\
\big|I_{n_{k-1}+m_k+t_k}(w^{(k)}, w^{(k,*)})\big|&\ge &   K^{-1} \eta\big|I_{n_{k-1}+m_k}(w^{(k)})\big|\cdot e^{-S_{n_{k-1}+m_k}f([w^{(k)}])}.\nonumber
\end{eqnarray}

Finally we introduce
  $$
\G_{k}\Big( \ep^{(k-1)} \Big)=\left\{(w^{(k)}, w^{(k, *)}): \ \begin{cases}  \  w^{(k)}=(\ep^{(k-1)}, w_1,\cdots, w_{\ell_k}) \\   \mbox{ with }w_i\in \Gamma_i ( \ep^{(k-1)}),    \ 1\le i\le \ell_k \end{cases}\right\},
$$ and
  $$
\F_{k}\Big( \ep^{(k-1)} \Big)=\left\{I_{n_{k-1}+m_k+t_k}(w^{(k)}, w^{(k, *)}): (w^{(k)}, w^{(k, *)})\in \G_k\Big( \ep^{(k-1)} \Big)\right\}.
$$

\begin{dfn}
The $k$th level of the Cantor set is defined as
$$
\F_k= \bigcup_{    \ep^{(k-1)} \in \G_{k-1}}   \F_{k}\Big( \ep^{(k-1)} \Big), \ \ \G_k= \bigcup_{    \ep^{(k-1)} \in \G_{k-1}}   \G_{k}\Big( \ep^{(k-1)} \Big).
$$\end{dfn}
%

\begin{rem}
 As we noticed for $\G_1$, it is important to remember that we omit some dependence in our notations. Every family of words $ \Gamma_\ell(\ep^{(k-1)})$ depends on the proceeding families  $ \Gamma_i(\ep^{(k-1)})$, $i\leq \ell-1$. Similarly, the integer $t_k$ depends on $w^{(k)}$.

\smallskip

As before, there is a uniform upper bound for the  integers $t_k$. By Lemma \ref{l2} and our choice \eqref{defwk} one has
$$
e^{-(n_{k-1}  +m_k) \|f\|_{\infty} }\le e^{-S_{n_{k-1}  +m_k }f([w^{(k)}])}\le |I_{t_k-1}\big(({w}^{(k)})^{\infty} \big)|\le K \rho^{t_k-1},
$$
which gives
$$
t_k\le - (n_{k-1}+m_k) \|f\|_\infty \log \rho +1.
$$
\end{rem}

\begin{rem}
Observe that, although the words of $\G_k$ do not have the same length, our choices \eqref{ff9bis} and \eqref{ff9} impose that for every $k\geq 2$, words belonging to $\G_k$ have lengths greater than all the words of $\G_{k-1}$.
\end{rem}

\subsubsection{The Cantor set, and its first property}

\begin{dfn}
The Cantor set $\F_{\infty}$ is defined as
 $$
%
\F_{\infty}=\bigcap_{k\ge 1}   \  \  \bigcup_{ I_{n_k}(\ep^{(k)}) \in \F_k}   \  \ I_{n_k}(\ep^{(k)})=\bigcap_{k\ge 1}   \  \  \bigcup_{ \ep^{(k)} \in \G_k}   \  \ I_{n_k}(\ep^{(k)}).
$$
\end{dfn}

Note that each word in $\G_{k-1}$ is the prefix of some word in $\G_k$. So we define $\G_{\infty}$ as the limit of the sequence of the families $\{\G_k\}_{k\ge 1}$. Then each word in $\G_{\infty}$ can be expressed as $$
[w^{(1)}_1, \cdots,w^{(1)}_{\ell_1}, w^{(1, *)},  w^{(2)}_1, \cdots,w^{(2)}_{\ell_2},w^{(2, *)},    \cdots, w_1^{(k)},\cdots, w_{\ell_k}^{(k)},w^{(k, *)},\cdots].
$$
We also write $\G_{\infty}$ formally as $$\G_{\infty}=\bigcap_{k\ge 1}   \  \  \bigcup_{ \ep^{(k)} \in \G_k}   \  \ \ep^{(k)}.
$$

 The first lemma shows that the set $\F_\infty$ is sitting on the right set of points.

\begin{lem}
One has
$\F_{\infty}\subset {R_2}(f)$.
\end{lem}

\begin{proof}
This is obvious since, by \eqref{b5}, any point $x\in \F_\infty$ belongs to an infinite number of cylinders $J_n(w)$.
\end{proof}

The next lemma says that a point $x\in \F_{\infty}$ may have multiple representations but it corresponds only one word in $\G_{\infty}$.
\begin{lem}\label{unique}
For any $x\in \F_{\infty}$, there exists a unique word $\ep^{(\infty)}\in \G_{\infty}$, such that $x=\pi(\ep^{(\infty)})$.
\end{lem}
\begin{proof}
This follows from the fact that in the construction of $\F_{\infty}$, the cylinders of a fixed generation of the Cantor set are well separated.  Indeed,  if $\ep_1^{(k)}$ and  $\ep_2^{(k)}$ belong to $\G_k$, either they have different  ``father" cylinders $ \ep_1^{(k-1)}$  and  $\ep_2^{(k-1)}$,  and Remark \ref{remdist} applies to the distance between cylinders of generation $k-1$, or they have the same father cylinder $\ep^{(k-1)}$ and    Remark \ref{remdist} applies with cylinders of generation $k$.
\end{proof}

At last, we give a notation: \begin{itemize}\item For each $\ep^{(\infty)}\in \G_{\infty}$ and $n\ge 1$, call $I_n(\ep^{(\infty)})$ a {\em basic cylinder} of order $n$.

\item For each $x\in \F_{\infty}$, if it corresponds to $\ep^{\infty}\in \G_{\infty}$, the cylinder containing $x$ is chosen to be the cylinder $I_n(\ep^{(\infty)})$, i.e. $I_n(x):=I_n(\ep^{(\infty)})$.
\medskip
\end{itemize}


\subsection{Supporting measure}

For any finite word $v\in \Lambda^*$, let $\Gamma(v)$ be the subfamily of $\Lambda^m$ chosen in Lemma \ref{l4} with respect to the word $v$. Then define $s=s_{m,v}$   as the  unique solution to the equation
 \begin{equation}\label{1.1}
\sum_{w\in \Gamma(v)}\left(\big|(T^m)'([w])\big|^{-1}\cdot e^{-S_mf([w])}\right)^s=1.
\end{equation}
We omit the dependence in $f$ for $s_{m,v}$ for brevity.

We will use this equality to spread the mass of a cylinder $I_n(v)$ to some of its sub-cylinders $I_{n+m}(vw)$.

\begin{lem}
\label{l5}
There exists a constant  $C$ independent of $\ep$ and $v$ such that if $s_{m,v}$ is defined by \eqref{1.1}, then
$$ 0\leq s_m(f) - s_{m,v} \leq  C \ep.$$
\end{lem}
\begin{proof}
We gather the information about the parameters. Fix one  finite word $v$ and the associated family $\Gamma (v)$. Let us denote   $\alpha_w = \big|(T^m)'([w])\big|^{-1}\cdot e^{-S_mf([w])} $ for every  $ w\in \Gamma(v)$.  We see that \eqref{1.1} implies
\begin{eqnarray*}
1 = \sum_{w\in \Gamma(v)} \alpha_w^{s_{m,v}} =  \sum_{w\in \Gamma(v)} \alpha_w^{s_{m}(f)}  \alpha_w^{s_{m,v}-s_m(f)} \geq \left(    \sum_{w\in \Gamma(v)} \alpha_w^{s_{m}(f)} \right) \min \{ \alpha_w^{s_{m,v}-s_m(f)} : w\in \Gamma(v)\}.
\end{eqnarray*}
Combining the last inequality with \eqref{minimi}, it follows that
\begin{eqnarray*}
 \max \{ \alpha_w^{s_m(f) -s_{m,v}} : w\in \Gamma(v) \}  \geq \widetilde \eta.
\end{eqnarray*}
By \eqref{ff3bis}, one has
$$ \left(K^{-2} e^{-3m\ep}\right)^{\frac{-\log \rho}{ 4\ep} (s_m(f) -s_{m,v})}
  \geq \widetilde \eta.$$
  Taking logarithm of both sides, one obtains
  $$0 \leq s_m(f) -s_{m,v} \leq  \ep      \frac{ 4 \log \widetilde \eta} { -\log \rho( -2\log K -3m\ep)} \leq  \ep      \frac{ 4 \log \widetilde \eta} {  \log \rho \log K  }.
  $$
  Hence the result follows with $C =   \frac{ 4 \log \widetilde \eta} {  \log \rho \log K  } >0$.
\end{proof}

\smallskip
As we remarked after Lemma \ref{unique},
for any $x\in \F_\infty$, let $I_n(x)=I_n(\epsilon^{(\infty)})$ where $\epsilon^{(\infty)}$ is the unique word in $G_\infty$ corresponding to $x$.

\smallskip

We are now going to construct  a measure supported on $\F_{\infty}$.
We write every     $x\in \F_{\infty}$ as
\begin{align*}
x &=[w^{(1)}_1, \ldots,w^{(1)}_{\ell_1}, w^{(1, *)},  w^{(2)}_1, \ldots,w^{(2)}_{\ell_2},w^{(2, *)},    \ldots, w_1^{(k)},\ldots, w_{\ell_k}^{(k)},w^{(k, *)},\ldots]
\\ &=[\ep^{(k-1)}, w_1^{(k)},\ldots, w_{\ell_k}^{(k)},w^{(k, *)},\ldots].\end{align*}

\subsubsection{Measure of $\mu(I_n(x)) $ when $n\leq n_1$.} \
 \label{sub1}
 \medskip

(a1)  Consider a word  $(w_1^{(1)},\ldots, w_{\ell_1}^{(1)}) $   in $\prod_{\ell=1}^{\ell_1} \, \Gamma_\ell$.  For every $2\leq \ell \leq \ell_1$, let  $s_\ell^{(1)}$ be  the solution to  (\ref{1.1}) with respect to the word $v=(w_1^{(1)},\ldots, w_{\ell-1}^{(1)})$. We set $s_1^{(1)}$ as the solution to  (\ref{1.1}) with respect to the word $v=\emptyset$.

When $n=\ell m$,  for the word $\widetilde w_\ell^{(1)}=(w_1^{(1)},\ldots, w_{\ell}^{(1)}) $ of length $\ell m$,  set
 $$
\mu(I_n( \widetilde w_\ell^{(1)} ))=\prod_{i=1}^{\ell} \left(\big|(T^m)'([  w_i^{(1)}])\big|^{-1} \cdot e^{-S_mf([  w_i^{(1)}])}\right)^{s_i^{(1)}}.
$$
This definition is consistent. More precisely, let $\Gamma_{\ell+1}:=\Gamma_{\ell+1}(\widetilde{w}_{\ell}^{(1)})$ defined by Lemma \ref{l4} with respect to $\widetilde{w}_{\ell}^{(1)}$. Then
\begin{eqnarray}
\nonumber
\sum_{ w\in \Gamma_{\ell+1}} \mu(I_n(\widetilde w_\ell^{(1)},w)) & = & \sum_{ w\in \Gamma_{\ell+1}}   \prod_{i=1}^{\ell} \left(\big|(T^m)'([w_i^{(1)}])\big|^{-1} \cdot e^{-S_mf([w_i^{(1)}])}\right)^{s_i^{(1)}}\\
\nonumber
&&\ \ \ \ \ \ \ \qquad \times\left(\big|(T^m)'([w])\big|^{-1} \cdot e^{-S_mf([w])}\right)^{s_{\ell+1}^{(1)}} \\
\nonumber
& = &   \prod_{i=1}^{\ell} \left(\big|(T^m)'([w_i^{(1)}])\big|^{-1} \cdot e^{-S_mf([w_i^{(1)}])}\right)^{s_i^{(1)}}\\
\label{eqq2}&& \ \ \ \ \ \ \ \qquad\times  \sum_{ w\in \Gamma_{\ell+1}} \left(\big|(T^m)'([w])\big|^{-1} \cdot e^{-S_mf([w])}\right)^{s_{\ell+1}^{(1)}} \\
\nonumber
& = &   \prod_{i=1}^{\ell} \left(\big|(T^m)'([w_i^{(1)}])\big|^{-1} \cdot e^{-S_mf([w_i^{(1)}])}\right)^{s_i^{(1)}} = \mu(I_n( \widetilde w_\ell^{(1)} )),
\end{eqnarray}
where we used that \eqref{eqq2} is equal to 1 by the definition of $s_{\ell+1}^{(1)}$ (\ref{1.1}).

\medskip

(a2)  When $n=\ell m+i$ with $0\le \ell<\ell_1$ and $0<i<m$, the measure of a cylinder associated with a word $w$ of length $n$ is simply defined as
$$
\mu(I_n(w) )=  \sum_{ \widetilde w^{ (1)}_{\ell+1} \in  \prod_{i=1}^{\ell+1} \Gamma_{i} : \,I_{(\ell+1)m}(\widetilde w^{ (1)}_{\ell+1})  \subset I_n(w)   } \ \mu \Big (I_{(\ell+1)m}(\widetilde w^{ (1)}_{\ell+1}) \Big),
$$ i.e. the total mass of its offspring of order $(\ell+1)m$.
This ensures  the consistency of our definition of the measure $\mu$ on all cylinders associated with words of length $\leq m_1$.

\medskip

(a3)  Now fix a word   $\widetilde{w}_{\ell_1}^{(1)}=(w_1^{(1)},\ldots, w_{\ell_1}^{(1)}) $   in $\prod_{\ell=1}^{\ell_1} \, \Gamma_\ell$, and consider the longer word   $(\widetilde{w}_{\ell_1}^{(1)}, w^{(1, *)})=(w^{(1)}, w^{(1,*)})$.  For every $m_1< n\le m_1 +t_1$, the measure $\mu$ will charge only one cylinder inside $I_{m_1}(w^{(1)})$, and thus the mass will stay the same. More precisely, for every $m_1< n\le m_1 +t_1$, we set
\begin{equation*}
\mu\Big (I_{n} (w^{(1)}, w^{(1,*)}) \Big)=\mu(I_{m_1}(w^{(1)}))=\prod_{i =1}^{\ell_1} \left(\big|(T^m)'([w_i^{(1)}])\big|^{-1} \cdot e^{-S_mf([w_i^{(1)}])}\right)^{s_i^{(1)}}.
\end{equation*}

This definition is consistent, because the cylinders $I_{\ell_1 m}(\widetilde{w}^{(1)})$  with $\widetilde{w}_{\ell_1}^{(1)}\in \prod_{\ell=1}^{\ell_1} \, \Gamma_\ell$ are disjoint and well separated.

\subsubsection{Measure of $\mu(I_n(x)) $ when $ n_{k-1} < n\leq n_k$.} \

 \medskip

Now we define the measure $\mu$ inductively, by using the same method as above.

\smallskip

Assume that for every word  $\ep^{(k-1)} \in \G_{k-1}$ the measure $\mu(I_n(\ep^{(k-1)} ))$ has been defined, for all $n\leq n_{k-1}$.
We explain the way of constructing the measure on finer cylinders.

\medskip

We fix  $\ep^{(k-1)} \in \G_{k-1}$.

\medskip

(b1)  Consider a word  $(w_1^{(k)},\ldots, w_{\ell_k}^{(k)}) $   in $\prod_{\ell=1}^{\ell_k} \, \Gamma_\ell(\ep^{(k-1)})$.  For every $2\leq \ell \leq \ell_k$, let  $s_\ell^{(k)}$ be  the solution to  (\ref{1.1}) with respect to the word $v=(\ep^{(k-1)}, w_1^{(k)},\ldots, w_{\ell-1}^{(k)})$. We set $s_1^{(k)}$ as the solution to  (\ref{1.1}) with respect to the word $v=\ep^{(k-1)}$.

When $n= n_{k-1} + \ell m$,  for the word $\widetilde w_\ell^{(k)}=(w_1^{(k)},\ldots, w_{\ell}^{(k)}) $ of length $\ell m$,  set
 \begin{equation}
 \label{defmuin}
\mu \Big (I_n(\ep^{(k-1)}, \widetilde w_\ell^{(k)} ) \Big)= \mu\Big(I_{n_{k-1}}(\ep^{(k-1)}) \Big)   \prod_{i=1}^{\ell} \left(\big|(T^m)'([  w_i^{(k)}])\big|^{-1} \cdot e^{-S_mf([  w_i^{(k)}])}\right)^{s_i^{(k)}}.
\end{equation}

This definition is consistent for the exact same reason as  in Subsection \ref{sub1}.
Let $\Gamma_{\ell+1}:=\Gamma_{\ell+1}(\ep^{(k-1)}, \widetilde{w}_{\ell}^{(k)})$ defined by Lemma \ref{l4} with 
respect to $(\ep^{(k-1)}, \widetilde{w}_{\ell}^{(k)})$. Then we have
\begin{align*}
&\sum_{ w\in \Gamma_{\ell+1} } \mu\Big(I_n(\ep^{(k-1)}, \widetilde w_\ell^{(k)},w)\Big)\\ = &\mu\Big(I_{n_{k-1}}(\ep^{(k-1)}) \Big)
\times  \sum_{ w\in \Gamma_{\ell+1} }   \prod_{i=1}^{\ell} \left(\big|(T^m)'([w_i^{(k)}])\big|^{-1} \cdot e^{-S_mf([w_i^{(k)}])}\right)^{s_i^{(k)}}\\
&\ \ \ \ \ \ \ \ \ \ \ \ \ \ \ \ \ \ \ \ \ \ \ \ \ \qquad \qquad\times\left(\big|(T^m)'([w])\big|^{-1} \cdot e^{-S_mf([w])}\right)^{s_{\ell+1}^{(k)}} \\
\nonumber
= &   \mu\Big(I_{n_{k-1}}(\ep^{(k-1)}) \Big)
\times \prod_{i=1}^{\ell} \left(\big|(T^m)'([w_i^{(k)}])\big|^{-1} \cdot e^{-S_mf([w_i^{(k)}])}\right)^{s_i^{(k1)}}\\
& \ \ \ \ \ \ \ \ \ \ \ \ \ \ \ \ \ \ \ \ \ \ \ \ \ \qquad \qquad\times  \sum_{ w\in \Gamma_{\ell+1} } \left(\big|(T^m)'([w])\big|^{-1} \cdot e^{-S_mf([w])}\right)^{s_{\ell+1}^{(k)}} \\
\nonumber
= &  \mu \Big (I_n(\ep^{(k-1)}, \widetilde w_\ell^{(k)} ) \Big),
\end{align*}
where for the last equality we used the definition of $s_{\ell+1}^{(k)}$ (\ref{1.1}).
\medskip

(b2)  When $n=n_{k-1}+ \ell m+i$ with $0\le \ell<\ell_k$ and $0<i<m$, the measure of a cylinder associated with a word $w$ of length $n$ is
$$
\mu(I_n(w) )=  \sum_{ \widetilde w^{ (k )}_{\ell+1} \in  \prod_{i=1}^{\ell+1} \Gamma_{i} (\ep^{(k-1)}) : \,I_{n_{k-1}+(\ell+1)m}(\ep^{(k-1)}, \widetilde w^{ (k)}_{\ell+1})  \subset I_n(w)   } \ \mu \Big (I_{n_{k-1}+(\ell+1)m}(\ep^{(k-1)}, \widetilde w^{ (k)}_{\ell+1}) \Big),
$$ i.e. the total mass of its offspring of order $n_{k-1}+(\ell+1)m$.
This ensures  the consistency of our definition of the measure $\mu$ on all cylinders included in $I_{n_{k-1}} (\ep^{(k-1)} )$  associated with words of length $\leq n_{k-1}+m_k$.

\medskip

(b3)  It remains us to take care of the words of length between $n_{k-1}+m_k$ and $n_k$.  Fix a word   $ \widetilde w_{\ell_k}^{(k)}=(w_1^{(k )},\ldots, w_{\ell_k}^{(k )}) $   in $\prod_{\ell=1}^{\ell_k} \, \Gamma_\ell    (\ep^{(k-1)})    $, and consider the longer word   $(\ep^{(k-1)}, \widetilde w_{\ell_k}^{(k)} , w^{(k,*)}) = (w^{(k)}, w^{(k,*)})$.  For every $n_{k-1}+ m_k< n\le n_{k-1} +m_k+t_k$, the measure $\mu$ will charge only one cylinder inside $I_{n_{k-1}+m_k}(   \ep^{(k-1)} , \widetilde w_{\ell_k}^{(k)})$:  for every $n_{k-1}+m_k< n\le n_{k-1}+m_k +t_k$, we set
\begin{equation*}
\mu\Big(I_{n}(w^{(k)}, w^{(k,*)})\Big)=\mu\Big(I_{n_{k-1}+m_k}(  \ep^{(k-1)} ,\widetilde w_{\ell_k}^{(k)}) \Big).
\end{equation*}

This definition is consistent, because the cylinders associated with the words $(\ep^{(k-1)} ,\widetilde w_{\ell_k}^{(k)}) $ with $ \widetilde w_{\ell_k}^{(k)} \in \prod_{\ell=1}^{\ell_k} \, \Gamma_\ell  (\ep^{(k-1)}) $ are disjoint and well separated.

  \subsubsection{Conclusion}
\

\medskip

The measure we have built satisfies  the   Kolmogorov's Consistency Condition, as we checked it along the construction through the definitions of the mass on the disjoint cylinders of each generation $\F_k$ of the Cantor set $\F_\infty$.  Hence, it can be uniquely extended into a Borel probability measure supported on $\F_{\infty}$.

\subsection{Diameters of basic cylinders}

Recall that  {\em basic cylinders} are those $I_n(\epsilon^{(\infty)})$ with $\epsilon^{(\infty)}\in \G_{\infty}$. Now we estimate their diameters.  Write  $x\in \F_{\infty}$ as
$$
x=[w^{(1)}_1, \ldots,w^{(1)}_{\ell_1}, w^{(1, *)}, \ldots, w_1^{(k)},\ldots, w_{\ell_k}^{(k)}, w^{(k, *)},\ldots]
$$
and as before for each $k\ge 1$,  let $$
w^{(k)}=[w^{(1)}_1, \ldots,w^{(1)}_{\ell_1}, w^{(1, *)}, \ldots, w_1^{(k)},\ldots, w_{\ell_k}^{(k)}]=[\ep^{(k-1)},  w_1^{(k)},\ldots, w_{\ell_k}^{(k)}].
$$ and $n_{k-1}$ the length of the word $\ep^{(k-1)}$ and $n_{k-1}+\ell_k m$ the length of the word $w^{(k)}$.

\medskip

{\bf (l1)  When $n=n_{k} $:}  By (\ref{ff6}) and \eqref{defwk}, we have
 \begin{align*}
&|I_{n_{k}}(x)|\ge  K^{-1} \eta  |I_{n_{k-1}+m_k}(x)|\cdot e^{-S_{n_{k-1}+m_k}f([w^{(k)}])}.
\end{align*}
 In addition, one has by Proposition \ref{p0}
   \begin{eqnarray*}
  |I_{n_{k-1}+m_k}(x)|  
  & \ge &  K^{-1}  |I_{n_{k-1}+(m_k -m)}(x)| \cdot |I_{m}(w_{\ell_k}^{(k)})| \\
  & \ge &  K^{-2}  |I_{n_{k-1}+(m_k -2m)}(x)| \cdot |I_{m}(w_{\ell_{k-1}}^{(k)})| \cdot |I_{m}(w_{\ell_k}^{(k)})| \\
  &\geq& ...\\
  & \geq & K^{-\ell_k}    |I_{n_{k-1} }(x)| \cdot \prod_{\ell=1}^{\ell_k}  |I_{m}(w_{\ell}^{(k)})|.
\end{eqnarray*}
By \eqref{2.4}, one also has for every $\ell$
 \begin{eqnarray*}
 |I_{m}(w_{\ell}^{(k)})| \geq K^{-1} |(T^m)'(w_{\ell}^{(k)})|^{-1}.
 \end{eqnarray*}
Moreover, by the tempered distortion \eqref{ff4} of $f$ , we have
$$ S_{n_{k-1}+m_k}f([w^{(k)}]) \leq  n_{k-1}\|f\|_{\infty} +S_{m_k} f([w_1^{(k)},\ldots,w_{\ell_k}^{(k)}]).$$
and
\begin{align*}
\Big| S_{ \ell_k m}f([w_1^{(k)},\ldots,w_{\ell_k}^{(k)}])   -\sum_{j=1}^{\ell_k } S_mf([w_j^{(k)}]) \Big| &\le \ell_k m \ep = m_k \ep.
\end{align*}
One deduces that
 \begin{align*}
  |I_{n_{k}}(x)|
  \geq &   \eta
  K^{-2\ell_k-1  }   |I_{n_{k-1} }(x)| \cdot \left(\prod_{\ell=1}^{\ell_k}    |(T^m)'(w_{\ell}^{(k)})|^{-1}  e^{-S_mf([w_j^{(k)} ])  }\right)\cdot e^{-n_{k-1}\|f\|_{\infty}  -m_k\ep}.
  \end{align*}
Recalling \eqref{ff9} and 
\eqref{ff3bis}, (\ref{ff10}),
one has
 \begin{eqnarray*}
  |I_{n_{k}}(x)|  & \ge   &      |I_{n_{k-1} }(x)| \cdot \left(\prod_{\ell=1}^{\ell_k}    |(T^m)'(w_{\ell})|^{-1}  e^{-S_mf([w_j^{(k)} ])  }\right) ^{1+{\frac{4\ep}{-\log \rho}}}.
  \end{eqnarray*}
 Then by iteration we arrive that
   \begin{align}\label{1.2}
|I_{n_k}(x)|&\ge
\prod_{j=1}^{k}\left(\prod_{\ell=1}^{\ell_j} \big|(T^m)'( [w_\ell^{(j)}])\big|^{-1}\cdot e^{-S_mf([w_\ell^{(j)}])}\right)^{1+{\frac{4\ep}{-\log \rho}}}.
\end{align}

\medskip

{\bf (l2)  When $n=n_{k-1} +\ell' m $, with $ 1\leq \ell' \leq \ell_k$:}   By Proposition \ref{p0}, we have
$$
|I_n(x)| \ge K^{-1}|I_{n_{{k-1} }}(x)| \cdot K^{-\ell'}\prod_{\ell =1}^{\ell'} \big|(T^m)([w_\ell^{(k)}])\big|^{-1}.$$
By \eqref{ff3bis}, one has for every $w\in \Lambda^m$
$$
  K^{-1}  \geq\left(\big|(T^m_w)'([w])\big|^{-1} \cdot e^{-S_mf([w])}\right)^{\frac{4\ep}{-\log \rho}}   e^{3m\ep}.
$$
Hence,  using \eqref{ff10}, one deduces that
\begin{align}
\nonumber
|I_n(x)|   \ge &\prod_{j=1}^{{k-1} }\left(\prod_{\ell=1}^{\ell_j} \big|(T^m)'( [w_\ell^{(j)}])\big|^{-1}\cdot e^{-S_mf([w_\ell^{(j}])}\right)^{1+{\frac{4\ep}{-\log \rho}}}\\
 \label{1.3}
&\times
\left(\prod_{\ell=1}^{\ell'} \big|(T^m)'([w_\ell^{(k)}])\big|^{-1} \cdot e^{-S_mf([w_\ell^{(k)}])}\right)^{1+\frac{4\ep}{-\log \rho}}.
\end{align}

\medskip

{\bf (l3) When $n=n_{{k-1} }+  \ell' m +i$ for $0\le \ell' < \ell_k$ and $1\le i<m$:}
In this case, we need only to see that
\begin{align}\label{1.4}
\eta_m \leq \frac{|I_n(x)|}{|I_{n_{{k-1} }+\ell 'm}(x)|} \leq 1.
\end{align}


\subsection{H\"{o}lder exponent of $\mu$}

We consider the measure $\mu$ on basic  cylinders in the construction of the Cantor set $\F_\infty$.  Let $n=n_{k-1}  + \ell' m $, with  $0\le \ell' < \ell_k$.  We use \eqref{defmuin}
 to see that
 \begin{eqnarray*}
 \mu \Big (I_n(x ) \Big)  
  &  =  & \prod_{i  =1}^{{k-1}  }    \prod_{\ell =1}^{\ell_i} \left(\big|(T^m)'([  w_\ell   ^{(i)}])\big|^{-1} \cdot e^{-S_mf([  w_\ell ^{(i)}])}\right)^{s_{\ell }^{(i)}}\\
 && \qquad \qquad \times  \prod_{\ell  =1}^{\ell  '}   \left(\big|(T^m)'([  w_\ell  ^{(k)}])\big|^{-1} \cdot e^{-S_mf([  w_\ell ^{(k)}])}\right)^{s_\ell ^{(k)}}.
 \end{eqnarray*}

 We apply \eqref{ff3} and Lemma \ref{l5} to see that  every real number $s_{\ell }^{(i)}$ appearing in the above product satisfies
$$|s_{\ell }^{(i)} - s(f)|  \leq |s_{\ell }^{(i)} - s_m(f)| +|s_m(f) -s(f)| \leq (C+1)\ep.$$

By \eqref{1.2} and \eqref{1.3} and the inequality above, one gets directly that
$$
\mu\Big(I_n(x)\Big)\le  |I_n(x) | ^{  s_\ep}, \ \ \ \mbox{ where } s_\ep:=   \frac{s(f)- (C+1)\ep}{1+\frac{4\ep}{-\log \rho}}.
$$

Up to a   constant  $M$ depending on the IFS  and the integer $m$ only, due to \eqref{1.4}, if  $n=n_{k-1}  + \ell' m +i$, with  $0\le \ell' < \ell_k$ and $0 \leq i \leq m$, one has
\begin{equation}
\label{eqfinal}
\mu\Big(I_n(x)\Big)\le M  |I_n(x) | ^{    s_\ep}.
\end{equation}

Finally, for the  cylinders $I_n(x) $ with $n_{k-1} + m_k<n\leq n_{k}-1$, the inequality is obvious since
$$\mu(I_n(x)) = \mu(I_{n_k}(x) )\leq  |I_{n_k}(x) | ^{  s_\ep} \leq  |I_n(x) | ^{  s_\ep}.
$$

\medskip

We have checked that \eqref{eqfinal} holds true on the basic cylinders appearing in the construction of the Cantor set $\F_\infty$. It remains us to cheek that it holds for all balls $B(x,r)\subset \zu^d$.

Consider one such ball $B(x,r)$ such that $\mu(B(x,r))>0$.  Let us denote by $k$ the unique generation such that $B(x,r)$ intersects at least two basic cylinders  of level $k$ in the construction of the Cantor set $\F_\infty$ and only one of level $k-1$.

 Now let $\ell$ be the largest integer such that $B(x,r)$ intersects  only one cylinder $I$ of order ${n_{k-1}+ \ell m}$. Observe that with this definition one may have $\ell \in \{0,1,...,\ell_k-1\}$. The maximality of $\ell$ ensures us that $B(x,r)$ intersects at least two sub-basic cylinders of order $n_{k-1}+(\ell+1)m$. Thus the diameter of the ball $B(x,r)$ must be larger than the gap between these  basic cylinders.

As a result, by Lemma \ref{l4} (see also Remark \ref{remark}), one has
$$|B(x,r)| \geq \eta_m |I| \ \ \mbox{ and } \ \ \mu(B(x,r)) \leq \mu(I) \leq M |I| ^{s_\ep},$$
thus
$$\mu(B(x,r)) \leq  (M \eta_m^{-s_\ep}) |B(x,r)| ^{s_\ep}.$$
In conclusion, \eqref{eqfinal}  holds true for any ball $B(x,r)\subset \zu^d$ up to the modification of the constant $M$ into $M \eta_m^{-s_\ep}$.

\medskip

The mass distribution principle yields that $$
\dim_{\textsf{H}}\F_{\infty}\ge s_\ep.
$$
Since $\dim_{\textsf{H}} {R}_2(f) \geq \dim_{\textsf{H}}\F_{\infty}$ and  $\lim_{\ep\to 0} s_\ep = s(f)$, the conclusion follows, and Theorem \ref{t2} is proved.

\section{Applications}
\label{section3}


 Roughly speaking, the set $R(f)$ concerns the distribution of periodic points, so Theorem \ref{t0} can be applied to study some Diophantine problems. In particular we apply Theorem \ref{t0} to prove a result concerning the approximation of reals
by quadratic algebraic numbers with purely periodic continued fraction expansions; when apply Theorem \ref{t0} to the $3$-adic expansion on triadic Cantor set, we can answer (partially) a question posed by K. Mahler \cite{Mah}.

\subsection{Recurrence properties in $b$-adic expansion}
Let $b\ge 2$ be an integer and $T$ the $b$-adic expansion given by $Tx=bx- \lfloor bx\rfloor$ for $x\in [0,1]$. Then the system $([0,1], T)$ can be viewed as a conformal IFS with $$\S=\Big\{\phi_i(x)=\frac{x+i}{b}: 0\le i<b\Big\}.$$

Following from Boshernitzan's result \cite{Bo}, for almost all $x\in [0,1]$, $$
\liminf_{n\to\infty} n|T^nx-x|<\infty.
$$
If we take $f(x) =t \log b$ as potential in Theorem \ref{t0}, one gets the
following from Theorem \ref{t0}:
\begin{thm} For any $t\ge 0$, the Hausdorff dimension of the set $$
\Big\{x\in [0,1]: |T^nx-x|<b^{-tn}, \ {\text{ for infinitely many }}\ n\in \N\Big\},
$$ is $1/(1+t)$.\end{thm}

It is immediate since in this case $\P(T, -s(\log |T'|+f)) = 1-s(1+t)$.

Let us take another example: Choosing $b=2$ and the potential $f$ equal to $f(x) = t \log b {\bf 1\!\! 1}_{[0,1/2)}(x)$,
 we get the following ``exotic" result, where the approximation rate of a point $x$ depends on the frequency of zeros in its dyadic decomposition.
\begin{thm} For every $x \in \zu $ with unique dyadic decomposition $x=x_1x_2 ...$ with $x_n\in \{0,1\}$, let $\xi_n(x) = \#\{ 1\leq i\leq n: x_i=0\}$ be the number of zeros amongst the first $n$ digits of $x$.
For any $t\ge 0$, the Hausdorff dimension of the set $$
\Big\{x\in [0,1]: |T^nx-x|<b^{-t \phi_n(x)}  \ {\text{ for infinitely many }}\ n\in \N\Big\},
$$ is the unique solution  to the equation
$1+2^{ts} = 2^{s(t+1)}$.
\end{thm}

An immediate computation shows that $$\P(T, -s(\log |T'|+f)) = \log \Big(2^{-s(t+1)} (1+2^{ts})\Big).$$

\subsection{Recurrence properties in continued fraction system}

The system of continued fraction is given by $$
T0=0, \ Tx=1/x-\lfloor 1/x\rfloor, \ x\in [0,1),
$$
It is a classic conformal iterated function system generated by $$
\S=\Big\{\phi_i(x)=1/(i+x): i\in \N\Big\}.
$$
%
  Theorem \ref{t0} applies to this system.

\subsection{Approximation by purely periodic quadratic numbers}

Let $\widetilde{\mathbb{A}}_2$ denote the class of quadratic algebraic numbers in $[0,1]$ with purely periodic continued fractions. Instead of approximating reals by rationals, we consider the approximation of reals by elements in $\widetilde{\mathbb{A}}_2$. 

The elements in $\widetilde{\mathbb{A}}_2$  are closely related to the reduced rationals in the following sense:

(i). For each $x\in \widetilde{\mathbb{A}}_2$, let $x=[(a_1,\ldots,a_n)^{\infty}]$ be its continued fraction expansion. With $x$ is naturally associated a    reduced rational number $p_n/q_n=[a_1,\ldots, a_n]$.

(ii). For each irreducible rational $p/q$, let $p/q=[a_1,\ldots, a_n]$ with $a_n\ge 2$ or $p/q=[a_1,\ldots,a_{n}-1, 1]$ be the two continued fraction expansion of $p/q$. Then the rational $p/q$ determines two elements in $\widetilde{\mathbb{A}}_2$, namely
$$x_1=[(a_1,\ldots,a_n)^{\infty}], \ x_2=[(a_1,\ldots,a_n-1,1)^{\infty}].$$
We call $x_1,x_2$ {\em the elements in $\widetilde{\mathbb{A}}_2$  induced by $p/q$.}

\medskip

For each $q\in \N$, let $$
\mathcal{A}_q=\Big\{x\in \widetilde{\mathbb{A}}_2: x \ {\text{is induced by }} p/q \ {\text{for some}} \ p\in \N \ {\text{with}} \ (p,q)=1\Big\}.
$$
Obviously, $\sharp \mathcal{A}_q\le 2q$.

\medskip

For any $\tau\ge 0$, we introduce the set
$$
D(\tau)=\Big\{x\in [0,1]: d(x, \mathcal{A}_q)<q^{-2(\tau+1)},  \ {\text{ for infinitely many }}\ q\in \N\Big\},
 $$
of real numbers approximated at a given rate by purely quadratic numbers. If   $\mathcal{A}_q$ is replaced by the set $\{p/q: 0\le p\le q\}$, then  the set $D(\tau)$ is just the classic Jarn\'{i}k-Besicovitch set.

\begin{thm}
For any $\tau\ge 2$, $\dim_{\textsf{H}}D(\tau)=1/(\tau+1)$.
\end{thm}
\begin{proof}  The upper bound of $\dim_{\textsf{H}}D(\tau)$ is obtained using the natural covering system of $D(\tau)$ and also noticing that
$\sharp \mathcal{A}_q\le 2q$.

\smallskip

For the lower bound, fix $\ep>0$ and consider   $f(x)=(\tau+\ep)\log |T'(x)|$ for any $\ep>0$ and the associated  set $R(f)$. Notice that $$
R(f)\subset\bigcap_{N=1}^{\infty}\bigcup_{n=N}^{\infty}\bigcup_{(a_1,\ldots,a_n)\in \N^n}J_n(a_1,\ldots,a_n),$$ where $$J_n(a_1,\ldots,a_n)=\Big\{x\in I_n(a_1,\ldots,a_n): |T^nx-x|<1/2 q_n(x)^{-2\tau}\Big\}.
$$
For each $(a_1,\ldots,a_n)\in \N^n$, let $x_0=[(a_1,\ldots,a_n)^{\infty}]$. Then $x_0\in\mathcal{A}_{q_n(x_0)}$.  For each $x\in J_n(a_1,\ldots,a_n)$, $q_n(x)=q_n(x_0)$. Then by the triangle inequality, we have \begin{align*}
|T^nx-x|&\ge |T^nx-T^nx_0|-|T^nx_0-x_0|-|x_0-x|\\
&=|(T^n)'(\xi)|\cdot |x-x_0|-|x-x_0|\\&\ge \frac{1}{2}q^2_n(x) |x-x_0|.
\end{align*} Thus $$
d(x, \mathcal{A}_{q_n(x)})\le q_n(x)^{-2(1+\tau)}.
$$ This shows that $R(f)\subset D(\tau)$.
\end{proof}

One can compare the dimension of $D(\tau)$ with that of the set of $\tau$-well approximable points by rationals \cite{Ja} and that of the set of $\tau$-well approximable points by quadratic algebraic numbers \cite{BaS1}.

\subsection{Recurrence properties in triadic Cantor set and a Mahler's question}

Let $C$ be the triadic Cantor set which $C$ can be viewed as the attractor of the finite iterated function system $$
\phi_1(x)=\frac{x}{3}, \ \ \phi_2(x)=\frac{2+x}{3}, \ x\in [0,1].
$$
Define the corresponding map $T:C\to C$ as $
Tx=3x \ {\text{mod}\ 1}.
$

Following from Boshernitzan's result \cite{Bo}, for almost all $x\in C$ with respect to the Cantor measure, $$
\liminf_{n\to\infty} n^{H }|T^nx-x|<\infty, \ \ \ \mbox{ where $H=\frac{\log 2}{\log 3}$}.
$$
While
following from Theorem \ref{t0}, we have
\begin{thm}\label{a1}
For any $t>0$, the Hausdorff dimension of the set
 $$
 {R}(f)=\Big\{x\in C: |T^nx-x|<3^{-tn}, \  \ {\text{ for infinitely many }}\ n\in \N\Big\},
$$ is $H/(t+1)$.
\end{thm}

This result can also be applied to answer a question posed by
K. Mahler \cite{Mah}: whether there exist well approximable points, except Liouville numbers, on triadic Cantor set? This was affirmatively answered by Levesley, Salp \& Velani \cite{LeSV}, Bugeaud \cite{Bu1} and independently by Barral \& Seuret \cite{BS1,BS2}. As far as the Hausdorff dimension of the set of the well approximable points in
 triadic Cantor sets is concerned, we can obtain the same result as in \cite{LeSV}:
\begin{cor}\label{c1}
The set of well approximable points in triadic Cantor set is of Hausdorff dimension at least $\frac{\log 2}{2\log 3}$.
\end{cor}
\begin{proof}  For any $t>0$, the set $R(f)$ in Theorem \ref{a1} can be rewritten as
$$R(f) =
\Big\{x\in C: \|(3^n-1)x\|<3^{-tn}, \ {\text{i.o.}}\ n\in \N\Big\}.
$$

When $t>1$,  $R(f)$ is a subset of $(t+1)$-well approximable points in $C$, and Theorem \ref{a1} yields  $\dim_{\textsf{H}} R(f) = H/(t+1)$. Letting $t$ tend to 1 yields the result.
\end{proof}

There is a small difference with   \cite{LeSV,Bu1,BS1,BS2}, where they restrict their attention to approximate the points in $C$ by rationals $p/q$ with $q\in \{3^n: n\in \N\}$, Corollary \ref{c1}
indicates that one can also use  rationals with periodic 3-adic expansion $$
\Big\{\frac{i_1}{3}+\ldots+\frac{i_n}{3^n}+\frac{i_1}{3^{n+1}}+\ldots+\frac{i_n}{3^{2n}}+\ldots: i_k=0 \ {\text{or}}\ 2, \ {\text{for all}}\ 1\le k\le n, n\in \N\Big\}
$$ to approximate the points in $C$ and the same dimensional result as in \cite{LeSV,Bu1} holds.


\medskip
\subsection*{Acknowledgements}

The first author would like to thank the members of the Mathematics Department of Huazhong University for their kind hospitality during his visit in November 2012.

 { }
\end{document}